\documentclass[11pt]{amsart}
\usepackage[english,activeacute]{babel}
\usepackage{amsmath,amsfonts,amssymb,amstext,amsthm,amscd,mathrsfs,amsbsy,xypic,centernot}
\usepackage{xcolor}
\usepackage{hyperref}

\newtheorem{theorem}{Theorem}[section]
\newtheorem{proposition}[theorem]{Proposition}

\newtheorem{lemma}[theorem]{Lemma}

\theoremstyle{definition}
\newtheorem{definition}[theorem]{Definition}
\newtheorem{example}[theorem]{Example}
\newtheorem{remark}[theorem]{Remark}
\newtheorem{convention}[theorem]{Convention}
\newtheorem{notation}[theorem]{Notation}

\newtheorem{theoremx}{Theorem}

\newcommand{\N}{\mathbb N}
\newcommand{\Z}{\mathbb Z}

\newcommand{\R}{\mathbb R}

\newcommand{\K}{\mathbb K}

\newcommand{\pol}{\mathbb \K[x_1,\ldots,x_n]}
\newcommand{\lau}{\mathbb \K[t_1^{\pm},\ldots,t_d^{\pm}]}
\newcommand{\kg}{\K[t^{\Gamma}]}

\newcommand{\V}{\mathbf{V}}

\newcommand{\sd}{\check{\sigma}}

\newcommand{\Rz}{\R_{\geq0}}

\newcommand{\conv}{\operatorname{Conv}}

\newcommand{\sing}{\operatorname{Sing}}
\newcommand{\ch}{\operatorname{char}}
\newcommand{\Gr}{\operatorname{Gr}}

\newcommand{\Jac}{\operatorname{Jac}}

\newcommand{\Bl}{\operatorname{Bl}}

\newcommand{\ns}{\mathcal{N}_{\sigma}}

\newcommand{\jp}{\mathcal{J}_p}
\newcommand{\jo}{\mathcal{J}_0}
\newcommand{\jj}{\mathcal{J}}

\newcommand{\ig}{I_{\Gamma}}
\newcommand{\xg}{X_{\Gamma}}

\providecommand{\keywords}[1]{{\textbf{Keywords:}} #1}

\begin{document}

\title{Nash blowups of toric varieties in prime characteristic}
 
\author[D. Duarte]{Daniel Duarte}
\address{Centro de Ciencias Matem\'aticas, UNAM, Campus Morelia, Morelia, Michoac\'an, M\'exico}
\email{adduarte@matmor.unam.mx}

\author[J. Jeffries]{Jack Jeffries$^\textrm{2}$}
\address{University of Nebraska-Lincoln}
\email{jack.jeffries@unl.edu}
\thanks{$^\textrm{2}$ The second author was partially supported by NSF CAREER Award DMS-2044833}

\author[L. N{\'u}{\~n}ez-Betancourt]{Luis N{\'u}{\~n}ez-Betancourt$^\textrm{3}$}
\address{Centro de Investigaci{\'o}n en Matem{\'a}ticas, Guanajuato, Gto., M{\'e}xico}
\email{luisnub@cimat.mx}
\thanks{$^\textrm{3}$ The third author was partially supported by  CONACyT Grant 284598 and C\'atedras Marcos Moshinsky.}
 



\begin{abstract} 
We initiate the study of the resolution of singularities properties of Nash blowups over fields of prime characteristic. We prove that the iteration of normalized Nash blowups desingularizes normal toric surfaces. We also introduce a prime characteristic version of the logarithmic Jacobian ideal of a toric variety and prove that its blowup coincides with the Nash blowup of the variety. As a consequence, the Nash blowup of a, not necessarily normal, toric variety of arbitrary dimension in prime characteristic can be described combinatorially.
\end{abstract}

\maketitle


\section*{Introduction}

The Nash blowup of an algebraic variety is a modification that replaces singular points by limits of tangent spaces. It has been proposed to solve singularities by iterating this blowup \cite{S,No}. This question has been extensively studied \cite{No,R,GS1,GS2,Hi,Sp,GS3,At,GM,GT,D1,Cha,DG}. 

Until recently, the resolution properties of Nash blowups have been studied only over fields of characteristic zero. This is due to a well-known example given by A. Nobile that shows that the Nash blowup could be trivial in prime characteristic \cite{No}. It was recently shown that Nash blowups behave as expected in prime characteristic after adding the condition of normality \cite{DN1}. Hence, the original question regarding Nash blowups and resolution of singularities can be reconsidered in arbitrary characteristic by iterating  normalized Nash blowups. We stress that the condition of normality is frequently assumed for many results also in characteristic zero. For instance, M. Spivakovsky showed that the iteration of normalized Nash blowups gives a resolution of singularities for complex surfaces \cite{Sp}. 

In this paper we initiate the study of the resolution properties of normalized Nash blowups in prime characteristic. We study this question in the context of toric varieties.

Our first goal in this paper is to describe combinatorially the Nash blowup of a not necessarily normal toric variety of arbitrary dimension over fields of prime characteristic. In characteristic zero, the key ingredient for such a description is given by the so-called logarithmic Jacobian ideal, which was originally introduced by Gonz\'alez-Sprinberg \cite{GS1}. This is a monomial ideal determined by linear relations on the generators of the semigroup of the toric variety. Gonz\'alez-Sprinberg showed that the Nash blowup of a normal toric variety is isomorphic to the blowup of its logarithmic Jacobian ideal. This result was later revisited and generalized by several authors. For instance, M. Lejeune-Jalabert and A. Reguera gave a new proof of this result \cite{LJ-R}. In addition, P. Gonz\'alez and B. Teissier extended this result for not necessarily normal toric varieties \cite{GT}.

It is worth mentioning that the isomorphism between Nash blowups of toric varieties and the blowup of their logarithmic Jacobian ideal does not hold over fields of prime characteristic (see Example \ref{counterex}). In this paper we define a positive characteristic version of the logarithmic Jacobian ideal (see Definition \ref{log jac mod p}) that allows us to extend the known result in characteristic zero to prime characteristic.

\begin{theoremx}[{Theorem \ref{Nash=log}}]\label{log=Nash}
The Nash blowup of a toric variety over a field of characteristic $p>0$ is isomorphic to the blowup of its logarithmic Jacobian ideal modulo $p$.
\end{theoremx}

Combining Theorem \ref{log=Nash} with previous work by  G\'onzalez-Teissier \cite{GT}, one obtains a combinatorial description of the Nash blowup in prime characteristic. We specialize this description to dimension two to obtain our second main theorem.

\begin{theoremx}[{Theorem \ref{Nash toric surfaces}}]\label{resol toric surf}
The iteration of normalized Nash blowups desingularizes normal toric surfaces over fields of prime characteristic.
\end{theoremx}

G. Gonz\'alez-Sprinberg proved the characteristic-zero version of Theorem \ref{resol toric surf} \cite{GS1}. An important step in his proof consists in showing that the normalized Nash blowup of a normal toric surface is determined by a subdivision of the cone defining the surface. The key step to prove Theorem \ref{resol toric surf} is to show that this subdivision is actually the same in prime characteristic. As a consequence, it is equivalent to solve normal toric surfaces iterating normalized Nash blowups in zero or prime characteristic. Once this is shown, the rest of the proof of  Theorem \ref{resol toric surf} follows the work of Gonz\'alez-Sprinberg.


As mentioned before,  the combinatorics of normalized Nash blowups of toric surfaces are independent of the characteristic. One may ask whether this is true for higher dimensions. We conclude this paper by exhibiting an example of a toric variety of dimension three whose Nash blowup behaves differently in characteristic zero and two. In particular, the normalized Nash blowup of this variety is nonsingular in characteristic zero but singular in characteristic two. A further iteration of normalized Nash blowup gives a nonsingular variety in this case.
\begin{convention}
Throughout this manuscript $\K$ denotes an algebraically closed field.
\end{convention}


\section{Nash blowups of toric varieties}

We are interested in studying the Nash blowup of a toric variety over fields of prime characteristic. We first recall some classical results of the characteristic zero case.

\begin{definition}
Let $\K$ be an algebraically closed field of arbitrary characteristic. Let $X\subseteq\K^n$ be an equidimensional algebraic variety of dimension~$d$. Consider the Gauss map:
\begin{align}
G:X\setminus\sing(X)&\rightarrow\Gr(d,n)\notag\\
x&\mapsto T_xX,\notag
\end{align}
where $\Gr(d,n)$ is the Grassmanian of $d$-dimensional vector spaces in $\K^n$, and $T_xX$ is the tangent space to $X$ at $x$. Denote by $X^*$ the Zariski closure of the graph of $G$. Call $\nu$ the restriction to $X^*$ of the projection of $X\times\Gr(d,n)$ to $X$. The pair $(X^*,\nu)$ is called the Nash blowup of $X$.
\end{definition}

\begin{definition}[{\cite{St,CLS}}]
Let $\Gamma\subseteq\Z^d$ be a semigroup generated by $\{\gamma_1,\ldots,\gamma_n\}$. Consider the $\K$-algebra homomorphism $\pi_{\Gamma}:\pol\to\lau$, $x_i\mapsto t^{\gamma_i}$. Let $\ig=\ker\pi_{\Gamma}$. The variety $\xg=\V(\ig)\subseteq\K^n$ is called the toric variety defined by $\Gamma$.  We denote as $\K[t^{\Gamma}]$ the coordinate ring of $\xg$.
\end{definition}

The first step towards a combinatorial description of Nash blowups of toric varieties in characteristic zero is given by the so-called logarithmic Jacobian ideal. This ideal was originally introduced by G. Gonz\'alez-Sprinberg \cite[Section 2]{GS1}, and later revisited by several authors \cite{LJ-R,GT,ChDG}.

\begin{definition}
Suppose that $\ch(\K)=0$. Let $\Gamma=\langle \gamma_1,\ldots,\gamma_n\rangle_{\N}\subseteq\Z^d$ be a semigroup such that $\langle \gamma_1,\ldots,\gamma_n\rangle_{\Z}=\Z^d$. Consider the following ideal:
$$\jo=\langle t^{\gamma_{i_1}+\cdots+\gamma_{i_d}}\,|\,\det(\gamma_{i_1}\cdots\gamma_{i_d})\neq0, \, 1\leq i_1<\cdots<i_d\leq n \rangle\subseteq\kg.$$
The ideal $\jo$ is called the \textit{logarithmic Jacobian ideal of $\xg$}.
\end{definition}

\begin{theorem}[{\cite{GS1,LJ-R,GT}}]\label{Nash=log-0}
Suppose that  $\ch(\K)=0$. The Nash blowup of $\xg$ is isomorphic to the blowup of its logarithmic Jacobian ideal.
\end{theorem}

The previous theorem is false over fields of prime characteristic. 

\begin{example}\label{counterex}
Suppose that  $\ch(\K)=2$, $\Gamma=\langle 2,3 \rangle_{\N}$, and $X_{\Gamma}={\V(x^3-y^2)}$. Then $\Bl_{\langle t^2,t^3\rangle}\xg$ is nonsingular but $\xg^*\cong\xg$ \cite[Example 1]{No}.
\end{example}

Our first goal is to give the positive characteristic version of Theorem \ref{Nash=log-0}.



\begin{definition}\label{log jac mod p}
Suppose that $\ch(\K)=p>0$. Let $\Gamma=\langle \gamma_1,\ldots,\gamma_n\rangle_{\N}\subseteq\Z^d$ be a semigroup such that $\langle \gamma_1,\ldots,\gamma_n\rangle_{\Z}=\Z^d$. Consider the following ideal:
$$\jp=\langle t^{\gamma_{i_1}+\cdots+\gamma_{i_d}}|\det(\gamma_{i_1}\cdots\gamma_{i_d})\neq0\, \mathrm{mod} \, p, 1\leq i_1<\cdots<i_d\leq n \rangle\subseteq\kg.$$
The ideal $\jp$ is called the \textit{logarithmic Jacobian ideal modulo $p$ of $\xg$}. 
\end{definition}

\begin{example}
Let $\Gamma=\langle 2,3\rangle_{\N}\subseteq\N$. Then $\mathcal{J}_2=\langle t^3 \rangle$, $\mathcal{J}_3=\langle t^2 \rangle$, and $\mathcal{J}_p=\langle t^2,t^3 \rangle$, for $p=0$ and $p\geq 5$.
\end{example}

We now present a property of short exact sequence of matrices. In fact,
there is a more general version of this property in the context of realizable matroids, which is known as Gale duality \cite[Theorem 2.2.8]{Ox}. 
We give the proof of this  property for the sake of completeness.

\begin{lemma}\label{exact}
Let
\begin{equation}
\xymatrix{0\ar[r]&\K^{c}\ar[r]^{B}&\K^{n}\ar[r]^A&\K^d\ar[r]&0}.\notag
\end{equation}
be a short exact sequence of vector spaces over $\K$.
Let $K\subseteq[n]=\{1,\ldots,n\}$, $|K|=c$. Denote as $B_K$ the matrix formed by the rows of $B$ corresponding to $K$. Similarly, denote as $A_{[n]\setminus K}$ the matrix formed by the columns of $A$ corresponding to $[n]\setminus K$. Then
$$\det B_K\neq0\Leftrightarrow \det A_{[n]\setminus K}\neq0.$$
\end{lemma}
\begin{proof}

Without loss of generality we assume $K=\{d+1,\ldots,n\}$. Denote $D=[n]\setminus K$. 






Assume that $\det(A_D)\neq 0$. Let $u\in \K^c$. Since multiplication by $B^T$ is surjective, there exists $\begin{pmatrix}
w_0 \\ v_0 \end{pmatrix}$ such that $B^T \begin{pmatrix}
w_0 \\ v_0 \end{pmatrix}= B^T_D w_0+B^T_K v_0=u.$
Since $\det(A^T_D)=\det(A_D)\neq 0$, there exists $z\in \K^d$ such that
$A^T_D z=w_0$. 
Let 
$v=v_0- A^T_K z $.
Thus, 
$\begin{pmatrix}
0 \\ v \end{pmatrix}=\begin{pmatrix}
w_0 \\ v_0 \end{pmatrix}-\begin{pmatrix}
A^T_D z \\ A^T_K z \end{pmatrix}
=\begin{pmatrix}
w_0 \\ v_0 \end{pmatrix}-A^Tz$.
Hence,
$$
B^T_Kv= 0+B^T_K v
=
B^T \begin{pmatrix}
0 \\ v \end{pmatrix}
=B^T\begin{pmatrix}
w_0 \\ v_0 \end{pmatrix}- B^T A^T z=B^T\begin{pmatrix}
w_0 \\ v_0 \end{pmatrix}=u.$$
Hence, the linear transformation associated  to $B^T_K$ is surjective, and so, 
$\det(B_K)=\det(B^T_K)\neq 0$.

The other implication follows as in the previous paragraph working on the dual exact sequence.
\end{proof}

\begin{theorem}\label{Nash=log}
Suppose that  $\ch(\K)=p>0$. The Nash blowup of $\xg$ is isomorphic to the blowup of its logarithmic Jacobian ideal modulo $p$.
\end{theorem}
\begin{proof}
The proof follows the arguments given by Gonz\'{a}lez P\'{e}rez and Teissier \cite[Proposition 60]{GT}, combined with the previous lemma.

Assume that $\{\gamma_1,\ldots,\gamma_n\}$ is a minimal generating set of $\Gamma\subseteq\Z^d$ and $\langle\gamma_1,\ldots,\gamma_n\rangle_{\Z}=\Z^d$. Let $A=(\gamma_1\cdots\gamma_n)_{d\times n}$ and $c=n-d$. Since $\K$ is algebraically closed, there exist $x^{\alpha_1}-x^{\beta_1},\ldots,x^{\alpha_c}-x^{\beta_c}\in\ig$ such that $\alpha_1-\beta_1,\ldots,\alpha_c-\beta_c$ define a basis of the kernel of $A$. Let $B=(b_1\cdots b_c)_{n\times c}$, where $b_i=\alpha_i-\beta_i$. Then we have an exact sequence,
\begin{equation}
\xymatrix{0\ar[r]&\Z^{c}\ar[r]^{B}&\Z^{n}\ar[r]^A&\Z^d\ar[r]&0}.\notag
\end{equation}
Since the kernel of $A$ is a saturated sublattice of $\Z^n$, one can extend $b_1,\ldots,b_c$ to a basis of $\Z^n$. We denote as $\bar{A}$ and $\bar{B}$ the matrices obtained by taking the entries modulo $p$. Hence, we obtain an exact sequence of $\Z_p$-vector spaces,
\begin{equation}\label{zp seq}
\xymatrix{0\ar[r]&\Z_p^{c}\ar[r]^{\bar{B}}&\Z_p^{n}\ar[r]^{\bar{A}}&\Z_p^d\ar[r]&0}.
\end{equation}
Take $D_0\subseteq\{1,\ldots,n\}$ such that $|D_0|=d$ and $\det(\bar{A}_{D_0})\neq0$. Letting $K_0=[n]\setminus D_0$, we get $\det(\bar{B}_{K_0})\neq0$ by Lemma \ref{exact}. Denote as $J_{K_0}$ the matrix formed by the columns corresponding to $K_0$ of the Jacobian matrix $\Jac(x^{\alpha_i}-x^{\beta_i})_{1\leq i\leq c}$.

We have that
$\prod_{k\in K_0}x_k\det J_{K_0}\equiv \prod_{i=1}^c x^{\alpha_i}\det\bar{B}_{K_0} \mod \ig$
\cite[Section 2.2, Lemme 2]{GS1}, \cite[Proposition 60]{GT}.
We point out that this congruence is usually proved over fields of characteristic zero; however,  the same proof holds in arbitrary characteristic.
Since $\ig$ is a prime ideal and $\{\gamma_1,\ldots,\gamma_n\}$ is a minimal generating set of $\Gamma$, we obtain $\det(J_{K_0})\neq0 \mod \ig$.  Hence, $\Jac(x^{\alpha_i}-x^{\beta_i})_{1\leq i\leq c}$ has maximal rank equal to $c$. It follows that the blowup of $\xg$ along the ideal $\langle \det J_K \, | \, {K\subseteq[n], |K|=c} \rangle$ is isomorphic to the Nash blowup of $\xg$ \cite[Theorem 1]{No}.

Let $K\subseteq[n]$, $|K|=c$. As before, $\prod_{k\in K}x_k \det J_K\equiv \prod_{i=1}^c x^{\alpha_i}\det \bar{B}_K\mod \ig.$ Multiply this congruence by the monomial $\prod_{i\in [n]\setminus K}x_i$ to obtain:
\begin{align}\label{cong2}
(x_1\cdots x_n)\det J_K\equiv \Big(\prod_{i\in [n]\setminus K}x_i \Big)\prod_{i=1}^c x^{\alpha_i}\det \bar{B}_K\mod \ig.
\end{align}

Since multiplying an ideal by a principal ideal gives isomorphic blowups, Congruence (\ref{cong2}) implies:
\begin{align}
\xg^*&\cong \Bl_{\langle\det J_K \, | \, {K\subseteq[n],|K|=c}\rangle}\xg\notag\\
&\cong \Bl_{\langle x_1\cdots x_n\rangle\langle \det J_K \, | \, {K\subseteq[n],|K|=c} \rangle}\xg \notag\\
&\cong \Bl_{\langle(\prod_{i\in [n]\setminus K}x_i) (\prod_{i=1}^c x^{\alpha_i})\det \bar{B}_K\, | \, {K\subseteq[n],|K|=c} \rangle}\xg\notag\\
&\cong \Bl_{\langle(\prod_{i\in [n]\setminus K}x_i)\det\bar{B}_K\, | \, {K\subseteq[n],|K|=c}\rangle}\xg. \notag
\end{align}
Using the isomorphism $\pol/\ig\cong\kg$ and Lemma \ref{exact} applied to the exact sequence (\ref{zp seq}), the ideal 
$$\Big\langle \Big(\prod_{i\in [n]\setminus K}x_i\Big)\det\bar{B}_K\, \Big| \, \substack{K\subseteq[n] \\ |K|=c} \Big\rangle=\Big\langle \prod_{i\in [n]\setminus K}x_i\, \Big|\, \det\bar{B}_K\neq0, \, {\substack{K\subseteq[n] \\ |K|=c} \Big\rangle}$$
corresponds to the ideal
$$\langle t^{\gamma_{i_1}+\cdots+\gamma_{i_d}}\, |\, \det(\gamma_{i_1}\cdots\gamma_{i_d})\neq0\, \mathrm{mod} \,  p, \, {1\leq i_1<\cdots<i_d\leq n} \rangle=\jp.$$
In conclusion, $\xg^*\cong \Bl_{\jp}\xg$.
\end{proof}

Using Theorem \ref{Nash=log}, a combinatorial description of the Nash blowup of a toric variety can be obtained with the  framework developed by G\'onzalez-Teissier for the blowup of a toric variety along any monomial ideal \cite[Section 2.6]{GT}. 

\begin{remark}
As mentioned in the introduction, the characteristic zero analogue of Theorem \ref{Nash=log} was studied by several authors. 
In particular, there is a version using the language of Minkowski sums \cite[Theorem 2.9]{At}. Even though the characteristic zero assumption is not explicitly stated in that work \cite[Theorem 2.9]{At}, it is implicitly used in the proof.
\end{remark}



\section{Resolution of normal toric surfaces by iterated normalized Nash blowups}

In this section we prove that normalized Nash blowups solve the singularities of normal toric surfaces over fields of prime characteristic. 


\begin{notation}\label{NotSec3}
Throughout this section,  $\sigma\subseteq\R^2$ denotes  a nonregular strongly convex rational polyhedral cone of dimension $2$, and $\Gamma=\sd\cap\Z^2$, where $\sd\subseteq\R^2$ is the dual cone of $\sigma$.
Let $\Theta\subseteq\R^2$ be the convex hull of $\Gamma\setminus\{(0,0)\}$. Let $\{\gamma_1,\ldots,\gamma_n\}\subseteq\Z^2$ be the points lying on the compact edges of the boundary polygon $\partial\Theta$. 
We order $\{\gamma_1,\ldots,\gamma_n\}$ according to the  counterclockwise order.
\end{notation}

We note that $n\geq3$ since $\sigma$ is nonregular. We now recall the following description of the minimal generating set of $\Gamma$.

\begin{proposition}[{\cite[Proposition 1.21]{O}}]\label{seg}
In the context of Notation~\ref{NotSec3}, $\{\gamma_1,\ldots,\gamma_n\}$ is the minimal set of generators of $\Gamma$.
\end{proposition}


We use Theorem \ref{Nash=log} to prove that iterated normalized Nash blowups solve normal toric surfaces in prime characteristic. We show that the involved combinatorics are independent of the characteristic. The result thus follows from the classical work of G. Gonz\'alez-Sprinberg \cite{GS1}.

\begin{lemma}\label{abg}
Let $\alpha,\beta,\gamma\in\Z^2_{\geq0}$ be pairwise linearly independent vectors, in counterclockwise orientation. For a subset $S\subseteq\R^2$, denote as $\conv(S)$ the convex hull of $S$.
\begin{enumerate}
\item If $\beta\in\conv(\{\alpha,\gamma\})$, then $\alpha+\gamma\in\conv(\{\alpha+\beta,\beta+\gamma\})$.
\item If $\beta\notin\conv(\{\alpha,\gamma\})+\Rz(\alpha,\gamma)$, then \[ \alpha+\gamma\in\conv(\{\alpha+\beta,\beta+\gamma\})+\Rz(\alpha,\gamma).\]
\end{enumerate}
\end{lemma}
\begin{proof}
The first statement of the lemma follows from direct computation.

We now prove (2). There exist unique $r,s\in\R$ such that $\beta=r\alpha+s\gamma$. We claim that $r,s>0$ and $0<r+s<1$. By the orientation, we have that $\beta$ is in the cone spanned by $\alpha$ and $\gamma$. Thus, $r,s\geq0$ and $r=0$ or $s=0$ contradicts linear independence. If $r+s\geq1$, let $r=r'+r''$ and $s=s'+s''$ with $r'+s'=1$ and $r',s',r'',s''\geq0$. Then
$$\beta=(r'\alpha+s'\gamma)+(r''\alpha+s''\gamma)\in\conv(\{\alpha,\gamma\})+\Rz(\alpha,\gamma),$$
which contradicts the hypothesis on $\beta$. This justifies the claim. Hence, $0<r,s<1$ and $0<r+s<1$.

We have $\alpha+\beta=(1+r)\alpha+s\gamma$ and $\beta+\gamma=r\alpha+(s+1)\gamma$. Set $u=1-r$ and $v=1-s$. Then $0<u,v<1$ and $1<u+v$. Write $u=u'+u''$ and $v=v'+v''$ with $u'+v'=1$ and $u',v',u'',v''\geq0$. Thus,
\begin{align}
\alpha+\gamma&=(\alpha+\gamma-\beta)+\beta\notag\\
&=(u\alpha+v\gamma)+(u'+v')\beta\notag\\
&=u'(\alpha+\beta)+v'(\beta+\gamma)+u''\alpha+v''\gamma\notag\\
&\in\conv(\{\alpha+\beta,\beta+\gamma\})+\Rz(\alpha,\gamma).\notag \qedhere
\end{align}
\end{proof}

\begin{proposition}\label{vertex}
In the setting of Notation \ref{NotSec3},
let $p=\ch(\K)$, not necessarily positive. Let $\ns(\jp)$ denote the convex hull of \[\{(\gamma_i+\gamma_j)+\sd\,|\,\det(\gamma_i\mbox{ }\gamma_j)\neq0\, \mathrm{mod} \, p\}_{1\leq i<j\leq n}.\] Then the vertices of $\ns(\jp)$ are contained in the set \[\{\gamma_\ell+\gamma_{\ell+1}\,| \, \ell \in\{1,\ldots,n-1\}\}.\] In particular, $\ns(\jp)=\ns(\jo)$, for any $p>0$.
\end{proposition}
\begin{proof}
As before, $\{\gamma_1,\ldots,\gamma_n\}$ denotes the minimal generating set of $\Gamma$, ordered counterclockwise. Recall that $n\geq3$. Up to a change of coordinates, we can assume that $\Gamma\subseteq\Z^2_
{\geq0}$. Let $1\leq i<j<k\leq n$. By Proposition \ref{seg}, we know that $\gamma_{i},\gamma_{j},\gamma_{k}$, satisfy the first and second conditions in Lemma \ref{abg}. By the lemma, 
$$\gamma_{i}+\gamma_{k}\in\conv(\{\gamma_{i}+\gamma_j,\gamma_j+\gamma_k\})\subseteq\ns(\jp),$$
or
$$\gamma_{i}+\gamma_{k}\in\conv(\{\gamma_{i}+\gamma_j,\gamma_j+\gamma_k\})+\Rz(\gamma_i,\gamma_k)\subseteq\ns(\jp).$$
Hence, $\gamma_i+\gamma_k$ is either on the facet determined by $\gamma_i+\gamma_j$ and $\gamma_j+\gamma_k$, or it is in the interior of $\ns(\jp)$. In particular, it is not a vertex of $\ns(\jp)$. We conclude that the vertices of $\ns(\jp)$ must be of the form $\gamma_\ell+\gamma_{\ell+1}$ for some $\ell\in\{1,\ldots,n-1\}$.

The last part of the proposition follows from the fact $\det(\gamma_{\ell}\mbox{ }\gamma_{l+1})=1$ for all $\ell\in\{1,\ldots,n-1\}$. 
\end{proof}

A similar statement to Proposition \ref{vertex} was already known  \cite[Proposition 4.11]{At}.

\begin{theorem}\label{Nash toric surfaces}
In the setting of Notation \ref{NotSec3}, suppose that  $\ch(\K)=p>0$. The iteration of Nash blowups followed by normalization solves the singularities of normal toric surfaces.
\end{theorem}
\begin{proof}
Recall the following classical theorem due to G. Gonz\'alez-Sprinberg: if $\ch(\K)=0$, normalized Nash blowups solves normal toric surfaces \cite[Section 2.3, Th\'eor\`eme]{GS1}. Equivalently, by Theorem \ref{Nash=log-0}, normalized blowups of the logarithmic Jacobian ideal $\jo$ solves normal toric surfaces. 

The normalized blowup of the logarithmic Jacobian ideal $\jo$ is determined by the Newton polygon $\ns(\jo)$, which, in turn, induces a subdivision of $\sigma$ \cite[Section 2.6 and Remark 26]{GT}. 
Hence, the iteration of normalized blowups of $\jo$ gives place to a regular subdivision of $\sigma$.

Now suppose $\ch(\K)=p>0$. By Proposition \ref{vertex}, $\ns(\jp)=\ns(\jo)$. In particular, the subdivision of $\sigma$ induced by $\ns(\jp)$ is the same as the one induced by $\ns(\jo)$. Hence, the iteration of normalized blowups of $\jp$ gives place to a regular subdivision of $\sigma$. By Theorem \ref{Nash=log}, we conclude that the iteration of normalized Nash blowups solves $\xg$.
\end{proof}


\section{Higher dimensional toric varieties}

In view of the results of previous sections, one may wonder whether the combinatorics of the Nash blowup of toric varieties are independent of the characteristic in general. In other words, me may ask whether the last statement of Proposition \ref{vertex} also holds for higher dimensional toric varieties. We conclude this paper by showing that this is not the case already in dimension three.

Let $\sigma\subseteq\R^3$ be a cone such that $\sd=\langle (1,0,0),(0,1,0),(1,1,2)\rangle_{\R_{\geq0}}\subseteq\R^3$. Using Macaulay2 \cite{M2}, we obtain that $\Gamma=\sd\cap\Z^3$ is minimally generated by $\gamma_1=(1,0,0)$, $\gamma_2=(0,1,0)$, $\gamma_3=(1,1,1)$, and $\gamma_4=(1,1,2)$. Hence, the logarithmic Jacobian ideals in characteristic zero and two are, respectively, $\jo=\langle t^{(2,2,1)},t^{(2,3,3)},t^{(3,2,3)},t^{(2,2,2)} \rangle$, $\jj_2=\langle t^{(2,2,1)},t^{(2,3,3)},t^{(3,2,3)}\rangle$.

Let $\ns(\jp)=\conv\{\gamma_i+\gamma_j+\gamma_k+\sd|\det(\gamma_i \mbox{ }\gamma_j \mbox{ }\gamma_k)\neq 0\mod p\}_{1\leq i<j<k\leq4}.$
We claim that $\ns(\jo)\neq\ns(\jj_2)$. We obtain ${\{x-2=0\}}$, ${\{y-2=0\}}$, and $\{x+y-z-2=0\}$ are supporting hyperplanes of $\ns(\jo)$ from a direct computation. The intersection of these three planes is $\{(2,2,2)\}$. Hence, $(2,2,2)$ is a vertex of $\ns(\jo)$. Finally, notice that $(2,2,2)\notin\ns(\jj_2)$.

Now we study the normalized Nash blowup of $\xg$. In characteristic zero, the normalization of $\xg^*$ is nonsingular \cite[Section 6, Table 5]{At}. 

To compute $\xg^*$ in characteristic two, we compute the blowup of $\jj_2$ following the combinatorial description for the blowup of a toric variety along any monomial ideal \cite[Section 2.6]{GT}. The polyhedron $\ns(\jj_2)$ has three vertices: $v_1=(2,2,1)$, $v_2=(3,2,3)$, $v_3=(2,3,3)$. These vertices give place to the affine charts of $\xg^*$. Each of these affine charts are also toric varieties, determined by the following semigroups:

\begin{align}
\Gamma_1=&\Gamma+\langle v_2-v_1,v_3-v_1 \rangle_{\N}\notag\\
&=\langle (1,0,0),(0,1,0),(1,1,1),(1,0,2),(0,1,2)\rangle_{\N},\notag\\
\Gamma_2=&\Gamma+\langle v_1-v_2,v_3-v_2 \rangle_{\N}\notag\\
&=\langle (1,0,0),(1,1,1),(1,1,2),(-1,0,-2),(-1,1,0)\rangle_{\N},\notag\\
\Gamma_3=&\Gamma+\langle v_1-v_3,v_2-v_3 \rangle_{\N}\notag\\
&=\langle (0,1,0),(1,1,1),(1,1,2),(0,-1,-2),(1,-1,0)\rangle_{\N}.\notag
\end{align}

To compute the normalization of $\xg^*$ we need to find the saturation of these semigroups. Using Macaulay2 we obtain the following minimal set of generators for the saturation of each $\Gamma_i$, denoted as $\overline{\Gamma_i}$:
\begin{align}
\overline{\Gamma_1}&=\langle (1,0,0),(0,1,0),(1,0,2),(0,1,2),(1,0,1),(0,1,1) \rangle_{\N},\notag\\
\overline{\Gamma_2}&=\langle (1,0,0),(-1,1,0),(0,0,-1),(-1,0,-2),(1,1,2),(0,1,1) \rangle_{\N},\notag\\
\overline{\Gamma_3}&=\langle (1,-1,0),(0,1,0),(0,0,-1),(0,-1,-2),(1,1,2),(1,0,1) \rangle_{\N}.\notag
\end{align}

In particular, the normalization of $\xg^*$ is singular. Repeating the entire algorithm for each $\overline{\Gamma_i}$, the resulting saturated semigroups can all be generated by three elements. We conclude that by iterating twice the normalized Nash blowup in characteristic two, we obtain a resolution of $\xg$.

We note also that the corresponding fans for the normalized blowups of $X_\Gamma$ with respect to $\jj_0$ and $\jj_2$  have different rays; in particular, neither blowup factors through the other via a small morphism.

\section*{Acknowledgments}

We thank the referees for their careful reading and valuable comments. The first author would like to thank Jawad Snoussi and Enrique Ch\'avez for stimulating discussions on Nash blowups of toric varieties.


\end{document}